\documentclass[12pt]{amsart}
\usepackage{epsfig, graphics, psfrag, color}

\theoremstyle{plain}
\newtheorem{theorem}{Theorem}[section]

\newtheorem*{namedtheorem}{\theoremname}
\newcommand{\theoremname}{testing}

\theoremstyle{definition}

\begin{document}

\title{What is an Almost Normal Surface?}

\author{Joel Hass}
\address{Department of Mathematics, University of California, Davis, California 95616}
\email{hass@math.ucdavis.edu}
\thanks{Partially supported by NSF grants DMS-0072348 and CCFÐ1117663}

%    General info
\date{July 1, 2012.}

\dedicatory{This paper is dedicated to Hyam Rubinstein on the occasion of his 60th birthday.}


\begin{abstract}
 A major breakthrough in the theory of topological algorithms occurred in 1992 when Hyam Rubinstein introduced the idea of an almost normal surface.
We explain how almost normal surfaces emerged naturally from the  study of  geodesics and   minimal surfaces. Patterns of stable and unstable geodesics can be used to  characterize the 2-sphere among surfaces, and similar patterns of normal and almost normal surfaces led Rubinstein to an algorithm for recognizing the 3-sphere.
\end{abstract}

\maketitle

\section{Normal Surfaces and Algorithms}

There is a long history of  interaction between low-dimensional topology and the theory of algorithms.
In 1910 Dehn posed the problem of finding an algorithm to 
recognize the unknot \cite{Dehn:10}. Dehn's approach was 
to check whether the fundamental group of the complement of the knot, for which 
a finite presentation can easily be computed, is infinite cyclic. This led Dehn to pose some of the first decision problems in
group theory, including asking for an algorithm 
to decide if a finitely presented group is infinite cyclic. It was shown about fifty years later that 
general group theory decision problems of this type are not decidable \cite{Stillwell}. 

Normal surfaces were introduced by Kneser  as a tool to describe and enumerate surfaces in a triangulated 3-manifold \cite{Knes:29}.  While a general surface inside a 3-dimensional manifold $M$ can be floppy, and have fingers and filligrees that wander around the manifold, the structure of a normal surface is locally restricted.  When viewed from within a single tetrahedron, normal surfaces look much like flat planes.  As with flat planes, they cross tetrahedra in collections of  triangles and quadrilaterals.  Each tetrahedron has seven types of elementary disks of this type;  four types of triangles and three types of quadrilaterals. The whole manifold has $7t$ elementary disk types, where $t$ is the number of 3-simplices in a triangulation.

\begin{figure}[htbp] %  figure placement: here, top, bottom, or page
   \centering 
 \includegraphics[width=2in]{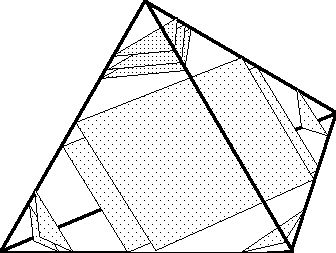} 
   \caption{A normal surface intersects a 3-simplex in triangles and quadrilaterals.}
   \label{normal}
\end{figure}

Kneser realized that the local rigidity of normal surfaces leads to finiteness results, and through them to the
Prime Decomposition Theorem for a 3-manifold. This theorem states that a 3-manifold can be cut open along finitely many  2-spheres into pieces that
are irreducible, after which the manifold cannot be cut further in a non-trivial way.  The idea behind this theorem is intuitively quite simple: if a very large number of disjoint surfaces are all uniformly flat, then some pair of the surfaces must be parallel.

A further advance came in the work of Haken, who gave the first algorithm for the unknotting problem  \cite{Hak:61}. 
Haken realized that a normal surface could be described by a vector with $7t$ integer entries,  with each entry describing the number of elementary disks of a given type.  Furthermore the matching of these disks across faces of a triangulation leads to a collection of integer linear equations, and this allows application of the techniques of integer linear programming.  In many
important cases, the search for surface that gives a solution to a topological  problem can be reduced to a search among a finite collection of
candidate surfaces, corresponding to a Hilbert Basis for the space of solutions to the equations \cite{HLP}.  Problems that can be solved algorithmically by this approach include:

\begin{tabbing}
{\em Problem:} UNKNOTTING\\
{\em INSTANCE:} A triangulated compact 3-dimensional manifold $M$ and a collection of edges $K$\\ in the 1-skeleton of $M$\\
{\em QUESTION:} \=  Does $K$ bound an embedded disk? \\
\end{tabbing} 

\begin{tabbing}
{\em Problem:} GENUS\\
{\em INSTANCE:} A triangulated compact 3-dimensional manifold $M$ and a collection of edges $K$\\ in the 1-skeleton of $M$ and an integer $g$\\
{\em QUESTION:} \=  Does $K$ bound an embedded surface of genus $g$? \\
\end{tabbing}

\begin{tabbing}
{\em Problem:} SPLITTING\\
{\em INSTANCE:} A triangulated compact  3-dimensional manifold $M$ and a collection of edges $K$\\ in the 1-skeleton of $M$\\
{\em QUESTION:} \=  Does $K$ have distinct components separated by an embedded sphere? \\
\end{tabbing}

But one major problem remained elusive.

\begin{tabbing}
{\em Problem:} 3-SPHERE RECOGNITION\\
{\em INSTANCE:} A triangulated 3-dimensional manifold $M$\\
{\em QUESTION:} \=  Is $M$ homeomorphic to the 3-sphere? \\
\end{tabbing}

Given Perelman's solution of the 3-dimensional
Poincare Conjecture \cite{Perelman}, we know that 3-Sphere Recognition is equivalent to the following.

\begin{tabbing}
{\em Problem:} SIMPLY CONNECTED 3-MANIFOLD\\
{\em INSTANCE:} A triangulated compact 3-dimensional manifold $M$\\
{\em QUESTION:} \=  Is  $M$ simply connected?
\end{tabbing}

The 3-Sphere recognition problem has important  consequences.
Note for example that the problem of deciding whether a given 4-dimensional simplicial complex has underlying space which is a manifold reduces to verifying that the link of each vertex is a 3-sphere, and thus to 3-Sphere Recognition.

In dimension two, the corresponding recognition problem is very easy.
Determining if a surface is homeomorphic to a 2-sphere can be solved by computing  its Euler characteristic. 
In contrast, for dimensions five and higher there is no algorithm to determine if a manifold is homeomorphic to  a sphere \cite{VolodinKuznetsovFomenko}, and
the status of the 4-sphere recognition problem remains open \cite{Nabutovsky}.
The related problem of fundamental group triviality is not decidable
 in manifolds of dimension four or higher. Until Rubinstein's work, there was no successful 
 approach to the triviality problem that took advantage of the special nature of 3-manifold groups.

For 3-sphere recognition one needs some computable way to characterize the 3-sphere. 
Unfortunately all 3-manifolds have zero Euler characteristic, and
no known easily computed invariant that can distinguish the 3-sphere among manifolds of dimension three.
Approaches  developed to characterize spheres in higher dimensions were based on simplifying some description, typically a Morse function.
The simplification process of a Morse function in dimension three, as given by a Heegaard splitting,
gets bogged down in complications.  Many attempts at 3-sphere recognition, if successful, 
imply combinatorial proofs of the Poincare Conjecture. Such combinatorial proofs have still not  been found.
A breakthrough occurred in the Spring of 1992, at a workshop at the Technion in Haifa, Israel.
Hyam Rubinstein presented a characterization of the 3-sphere that was suitable to algorithmic analysis.  In a series of talks at this
workshop he introduced a new algorithm that takes a triangulated 3-manifold and determines whether it is a 3-sphere.
The key new concept was an {\em almost normal surface}.

\section{What is an almost normal surface?}

Almost normal surfaces, as with their normal relatives, intersect each 3-simplex in $M$ in a collection of  triangles or quadrilaterals, with one exception.
In a single 3-simplex the intersection with the almost normal surface consists of either an an octagon or a pair of normal disks connected by a tube, as shown in Figure ~\ref{almostnormal}.

\begin{figure}[htbp] %  figure placement: here, top, bottom, or page
   \centering
 \includegraphics[width=2in]{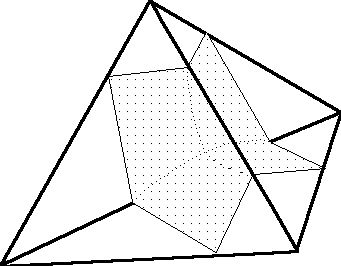} 
 \includegraphics[width=2in]{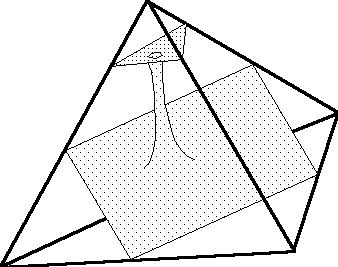} 
   \caption{Almost normal surfaces intersect  one 3-simplex in an octagon, or two normal disks tubed together.}
   \label{almostnormal}
\end{figure}

Rubinstein argued that an almost normal 2-sphere  had to occur in any triangulation of a 3-sphere, and in fact that the search for the 
presence or absence of this almost normal 2-sphere could be used to build an algorithm to recognize the 3-sphere.  Shortly afterwards, Abigail Thompson combined Rubinstein's ideas with techniques from the theory of thin position of knots, and gave an alternate approach to proving that Rubinstein's algorithm was valid \cite{Thompson}.  The question we address here is the geometrical background that motivated Rubinstein's breakthrough.

To describe the ideas from which almost normal surfaces emerged, we take a diversion into differential geometry and some results in the theory of geodesics and minimal surfaces.
A classical problem asks which surfaces contain closed, embedded (or simple) geodesics. The problem is hardest for a 2-sphere, since for other surfaces  a shortest  closed curve that is not homotopic to a point gives an embedded geodesic.
A series of results going back to Poincare establishes that every 2-sphere contains a simple closed geodesic \cite{ Croke, Grayson, HassScott, Poincare, Klingenberg}.
In fact any 2-sphere always contains no less than three simple, closed and {\em unstable} geodesics.  
Unstable means that while each sufficiently short arc of the geodesic minimizes length among curves connecting its endpoints, the entire curve can be pushed to either side in a manner that decreases length. 
The classic example is an equator of a round sphere, for which a sub-arc of length shorter than
$\pi$ is length minimizing, whereas longer arcs can be shortened by a deformation, as can the whole curve.
In Figure~\ref{unstable} we show several differently shaped 2-spheres and indicate unstable geodesics on each of them.

\begin{figure}[htbp] %  figure placement: here, top, bottom, or page
   \centering
\includegraphics[width=4in]{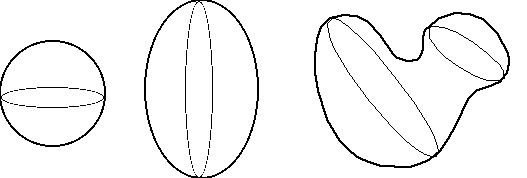} 
   \caption{Some unstable geodesics on 2-spheres of various shapes}
   \label{unstable}
\end{figure}

A conceptually simple argument shows that  unstable geodesics exist for any Riemannian metric on a 2-sphere, using a minimax argument that goes back at least to Birkhoff \cite{Birkhoff}.  Starting with a very short curve, drag it over the 2-sphere until it shrinks
to a point on the other side.  Among all such families of curves, look at the family whose longest curve is as short as possible.  This minimax curve provides an unstable geodesic.  It is not hard to show such a curve exists.

 Surfaces other then the 2-sphere do not necessarily contain an unstable geodesic.  The torus has a flat metric and higher genus surfaces have hyperbolic metrics, and in these metrics there are no unstable geodesics.  Even the projective plane, the  closest geometric relative of the 2-sphere, has no unstable geodesics in its elliptic metric.  Therefore the property of always having an unstable geodesic, for any metric, characterizes the 2-sphere.  

We will need to refine this to develop an algorithm. Any surface has some metrics in which there are both stable and unstable geodesics.
So given any fixed Riemannian metric on a surface, we focus on a maximal collection of disjoint 
separating geodesics, both stable and unstable.  See Figure~\ref{patterns}, where 
unstable geodesics are drawn as solid curves and  stable geodesics as dashed curves.
We assume a ``generic'' metric on a surface, in which there are only finitely many disjoint geodesics.   Almost all metrics have this
property, which can be achieved by a small perturbation of any metric  \cite{White}. 

\begin{figure}[htbp] %  figure placement: here, top, bottom, or page
   \centering
 \includegraphics[width=2.5in]{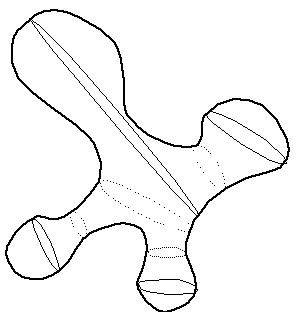} 
  \includegraphics[width=2.5in]{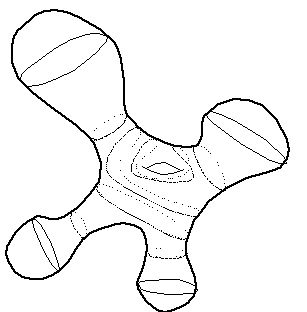} 
   \caption{Maximal collections of disjoint separating geodesics on a 2-sphere and a torus. Stable geodesics are shown with broken curves.}
   \label{patterns}
\end{figure}

In these examples we see certain patterns among a maximal collection of disjoint geodesics on a 2-sphere. These are summarized in the following result.

\begin{theorem}\label{geod.patterns}
Let $F$ be an orientable surface with a generic metric and ${\mathcal G}$ 
a maximal collection of disjoint, simple, closed and separating geodesics on $F$. Then ${\mathcal G}$  has the following properties.
\end{theorem} 
\begin{itemize}
\item If $F$ is a 2-sphere then ${\mathcal G}$  contains an unstable geodesic.  
\item A region  in $F-{\mathcal G}$ whose boundary is a single unstable geodesic is a disk.
\item A region  in $F-{\mathcal G}$ whose boundary is a single stable geodesic is a punctured torus.
\item A region in $F-{\mathcal G}$ with two boundary geodesics is an annulus whose boundary consists of one stable and one unstable geodesic.
\item A region in $F-{\mathcal G}$ with three boundary geodesics is a ``pair of pants''  whose boundary consists of  three stable geodesics.
\item No region of $F-{\mathcal G}$ has four or more boundary geodesics.
\end{itemize}
\begin{proof}
The proof applies minimax arguments using the curvature flow techniques developed by Gage, Hamilton, and Grayson \cite{Grayson2}.
The curvature flow deforms a curve on a smooth Riemannian surface in the direction of its curvature vector. Applying this flow to a family of curves gives a continuous deformation of the entire family, and decreases the length of each of curve, limiting to a point or a geodesic \cite{Grayson}.

If a region has an unstable geodesic on its boundary, then this boundary curve can be pushed in slightly and then shrunk by the curvature flow until it converges to a stable geodesic or to a point. Thus each region with an unstable geodesic on its boundary is either a disk or an annulus bounded by one stable and one unstable geodesic. The boundary curve of a complementary disk region must be unstable, since shrinking a stable boundary geodesic to a point gives a family of curves in the disk whose minimax curve is an unstable geodesic in the interior of the disk. But complementary regions contain no interior geodesics.

A region bounded by a single stable geodesic cannot contain a separating essential curve that is not boundary parallel, since such a curve could be homotoped to a separating geodesic in the interior of the region. Thus all essential, non-boundary parallel simple closed curves in the region are non-separating.  Such a curve must exist since the region is not a disk, and so the region must be a punctured torus.

A minimax argument shows that an annular region bounded by two stable geodesics has an unstable geodesic separating its two boundary geodesics.  The maximality of ${\mathcal G}$ rules out this configuration.

If a region has two non-homotopic stable geodesics on its boundary, then we can find a new closed  separating curve by tubing the two boundary geodesics along a shortest arc connecting them within the region.  This new curve can be shortened within the region till it converges to a third stable geodesic, which must be a third boundary component.  Thus the region is a pair of pants and has exactly three stable geodesics on its boundary.  It follows that  no region has more than three boundary geodesics.
\end{proof}

These patterns can be used to distinguish the 2-sphere from other surfaces.  Fix any generic metric on a surface $F$ and let ${\mathcal G}$ be a maximal family of separating, simple, disjoint geodesics. 
 
\begin{theorem}[Geometric 2-Sphere Characterization]
  $F$ is a 2-sphere if $\mathcal G$ satisfies the following conditions:
  \begin{itemize}
\item  There is at least one unstable geodesic in  $\mathcal G$.   
\item No complementary region of $ F- \mathcal G$  has boundary consisting of a single stable geodesic.
 \end{itemize}
\end{theorem}
\begin{proof}
Suppose that $F$ satisfies these two conditions.  Pushing the unstable geodesic to either side decreases its length.  Continuing to decrease length with the curvature flow, we arrive either at a stable geodesic or a point.  If we arrive at a point then the unstable geodesic bounds a disk on that side. If we arrive at a stable geodesic then we consider the region on its other side.  If this region has only one boundary component  then the surface is not a sphere since it contains a punctured torus. If the region has one other unstable boundary curve then it is an annulus.
If the region has more than two stable boundary curves, then it's a pair of pants with three stable boundary geodesics.  Continuing across the new boundary geodesics, we construct a surface from pieces whose dual graph forms a tree.  Unless we encounter a complementary region of $ F- \mathcal G$  whose boundary has exactly one stable geodesic, the surface $F$ is a union of annuli, pairs of pants and disks, and these form a 2-sphere.  
\end{proof}

A very similar characterization carries over to dimension three and forms the basis of Rubinstein's 3-sphere recognition algorithm.  We first address the restriction of the curves we considered above  to separating curves.
One can distinguish separating and non-separating curves on a surface with homology, and homology can be efficiently computed from the simplicial structure of a triangulated manifold. Thus in searching for the 3-sphere we can immediately rule out any manifold that does not have the same homology as the 3-sphere. In a homology 3-sphere, every surface separates.  In dimension two, homology itself is enough to characterize the 2-sphere, though we did not take advantage of this in our construction. In dimension three, homology computations alone do not characterize the 3-sphere, but do reduce the candidates to the class of homology 3-spheres.  So we can assume that we are working in this class and that all surfaces are separating.  In particular we can rule out the possibility that $M$ contains a non-separating sphere or an embedded projective plane.

For a characterization of the 3-sphere we look at stable and unstable minimal surfaces instead of geodesics.  
By 1991 Rubinstein had made two important contributions to the study of such minimal surfaces
in dimension three.  Each of these two contributions played a key role in the creation of  the 3-sphere recognition algorithm.
 
Rubinstein had worked on the highly non-trivial problem of showing the existence of minimal representatives for various classes of surfaces in 3-manifolds. Simon and Smith had shown that the 3-sphere, with any Riemannian metric, contains an embedded minimal 2-sphere \cite{SimonSmith}.
This result was extended by Jost and by  Pitts and Rubinstein \cite{Jost, PittsRubinstein}. In a series of papers Pitts and Rubinstein developed a 
program which showed that a very large class of surfaces in 3-manifolds can be isotoped to be minimal. In particular, their methods indicated that a strongly irreducible Heegaard splitting in a 3-manifold always has an unstable minimal representative.
To show that a 3-sphere, with any Riemannian metric, contains an unstable minimal 2-sphere, start with a tiny 2-sphere and drag it over the 3-sphere until it shrinks down to a point on the other side.  Among all such families look for the biggest area 2-sphere in the family and choose a family that makes this area as small as possible. This minimax construction gives an unstable minimal 2-sphere.  The existence proof is more subtle than for a geodesic, but the concepts are similar, and the method extends to give the following insight.  Suppose we take a stable minimal 2-sphere in a 3-sphere and shrink it to a point, after necessarily first enlarging its area. Then among all such families of 2-spheres there is one whose largest area sphere has smallest area. This minimax 2-sphere is an unstable minimal 2-sphere.

The methods of Pitts-Rubinstein  can be used to characterize the 3-ball, similarly to the first two conditions of Theorem~\ref{geod.patterns}. The theory is considerably harder since there is no simple surface flow available to decrease area, unlike the curvature flow for curves in dimension two. Moreover spheres can split into several components as their area decreases, unlike curves. However these difficulties can be overcome  \cite{PittsRubinstein, Jost, SimonSmith}.

Suppose $B$  is a 3-manifold:\\ \\
{\bf Geometric 3-Ball Characterization:}  \\
{\em $B$ is a 3-ball if it satisfies the following conditions}
\begin{itemize}
\item  The boundary of  $B$  is a stable minimal 2-sphere.
\item  The interior of  $B$  contains no stable minimal 2-sphere.
\item The interior of  $B$  contains an unstable minimal 2-sphere.
 \end{itemize}

The idea of such a 3-Ball Characterization follows the lines of the two-dimensional case.  Suppose that $B$ satisfies the three assumptions.
Then $B$ contains an unstable minimal 2-sphere in its interior.  Shrinking this 2-sphere to one side must move it to $\partial B$, as otherwise it would get stuck on some stable minimal 2-sphere in the interior of $B$.  Similarly, shrinking this 2-sphere to the other side must collapse it to a point, or again it would get stuck on a stable minimal 2-sphere in the interior of $B$.  Thus $B$ is swept out by embedded spheres and homeomorphic to a ball.
 
A similar result characterizes the 3-sphere. Let ${\mathcal S}$ be a maximal family of separating disjoint embedded minimal spheres in $M$, both stable and unstable.  We are assuming that $M$ is a homology sphere, so all surfaces separate.

For a generic metric on a 3-manifold $M$, the collection of disjoint minimal spheres ${\mathcal S}$ is finite.  If $M$ contains infinitely many disjoint minimal spheres, then they can be used to partition $M$ into infinitely many components.  In each component one can find an embedded stable minimal sphere by applying the method of Meeks-Simon-Yau \cite{MeeksSimonYau}. But stable minimal spheres in $M$ satisfy uniform bounds on their second fundamental form \cite[Theorem 3] {Schoen}, implying a lower bound to the volume between two such spheres unless they are parallel (meaning that each projects homeomorphically to the other under the nearest point projection).  An infinite sequence of parallel minimal 2-spheres has a subsequence converging to a minimal 2-sphere with a Jacobi Field. But a theorem of White gives the absence of Jacobi fields for a minimal surface in a generic metric \cite{White}.  
\\ \\
{\bf Geometric 3-Sphere Characterization:}  \\
{\em $M$ is a 3-sphere if and only if  no complementary region of $ M - \mathcal S$  has boundary consisting entirely of stable minimal 2-spheres.}
 
 \begin{proof}
First note that $M$ is homeomorphic to a 3-sphere if and only if every complementary component $X$ of  $ M - \mathcal S$ is a punctured ball.

Suppose that $X$ is a complementary component of $ M - \mathcal S$   and consider the case where $X$ has an unstable minimal 2-sphere $\Sigma$  among its boundary components.
Then we can push   $\Sigma$   in slightly and apply the theorem of Meeks-Simon and Yau to minimize in its isotopy class \cite{MeeksSimonYau}. This gives a collection of
stable minimal 2-spheres, that, when  joined by tubes, recover the isotopy class of $\Sigma$. We conclude that $X$ is a punctured ball with exactly one unstable boundary component.

Now suppose that $X$ has all its boundary components stable.  We will show by contradiction that $X$ is not a punctured ball.  If it were, then 
it could be swept out by a family of 2-spheres.  This family begins with a 2-sphere that tubes together all the boundary 2-spheres of $X$ and
ends at a point.  By the methods of Simon and Smith \cite{SimonSmith}, see also  \cite{Jost, PittsRubinstein}, we obtain an unstable minimal 2-sphere in the interior of $X$.  But this contradicts maximality of $\mathcal S$, so $X$ cannot be a punctured ball.

Together, these cases give the desired characterization.
 \end{proof}
 
To translate the geometric characterization into an algorithm, we need a corresponding combinatorial theory that characterizes the 3-sphere among triangulated 3-manifolds.   
We need to replace the ideas of Riemannian 
geometry with PL versions that capture the relevant ideas.  Fortunately, natural PL-approximations to length and area
exist in dimensions two and three.  
Length is approximated by the {\em weight}, which measures how many times a curve crosses the edges of a triangulation, and area by
how many times a surface intersects edges. 
Combinatorial length and area can be related to Riemannian area by taking a series of metrics whose limit has support on the 1-skeleton. 

For curves on a surface, the analog of a geodesic then becomes a  special type  of {\em  normal curve}.  A normal curve intersects each two-simplex in arcs joining distinct edges of the two-simplex, so that no arc doubles back and has both endpoints on the same edge. A stable PL-geodesic is defined to be  a normal curve for which any deformation increases weight.  For deformations we allow isotopies of the curve in the surface which are non-transverse to edges or vertices at finitely many times. An unstable PL-geodesic is a normal curve that admits a weight decreasing deformation to each of its two sides.  Note that not all normal curves are PL-geodesics. In the triangulation of the 2-sphere given by a tetrahedron, there are three unstable PL-geodesics given by quadrilaterals, and additional unstable PL-geodesics of weight eight and above.  A curve of weight three surrounding a vertex is a normal curve, but not a PL-geodesic. See Figure~\ref{normalcurves}.

\begin{figure}[htbp] %  figure placement: here, top, bottom, or page
   \centering
 \includegraphics[width=2.in]{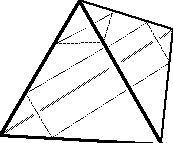} 
   \caption{A length four normal curve forms an unstable PL-geodesic.}
   \label{normalcurves}
\end{figure}

The analogous combinatorial area for surfaces in triangulated 3-manifolds theory was investigated in a series of papers by Jaco and Rubinstein. In their work on PL-minimal surfaces,  Jaco and Rubinstein showed that many of the properties that made minimal surfaces so useful in studying 3-manifolds still held when using combinatorial area \cite{JacoRubinstein}. 
For surfaces in 3-manifolds and deformations of these surfaces that avoid vertices, normal surfaces play the role of stable minimal surfaces. The question of which surfaces take the role of unstable minimal surfaces in the combinatorial theory was unclear until Rubinstein's insight that almost normal surfaces fill this role.  Just as unstable geodesics can be pushed to either side so as to decrease length, and unstable minimal surfaces can be pushed to either side to decrease area, so almost normal surfaces can be pushed to either side so as to decrease weight, or combinatorial area.

These two ingredients, the existence of unstable minimal surfaces and the construction of combinatorial versions of stable and unstable minimal surfaces, combine to give an algorithm to recognize the 3-sphere. The characterization of a 3-sphere via its minimal surfaces can be turned into a characterization via properties of  piecewise linear surfaces, properties that can be determined by constructing and examining a finite collection of normal and almost normal surfaces. 

\section{Recognizing the 3-sphere}
 
Rubinstein's algorithm is essentially the PL version of the geometric 3-sphere characterization given above. The characterization of the 3-sphere among triangulated homology 3-spheres begins by computing a maximal family ${\mathcal S}$ of disjoint normal and almost normal  spheres in a candidate manifold $M$.\\ \\
{\bf 3-Sphere Characterization:}  \\
{\em $M$ is a 3-sphere if $\mathcal S$ satisfies the following conditions:}
\begin{itemize}
\item  There is at least one almost normal sphere in  $\mathcal S$.   
\item No complementary region of $ M - \mathcal S$  has boundary consisting of a single normal sphere, other than a neighborhood of a vertex.
 \end{itemize}
These conditions can be checked by a finite procedure, and so give an algorithm.

The algorithm for recognizing the 3-sphere proceeds as follows.
One begins with a collection of 3-simplices and instructions for identifying their faces in pairs.
\begin{itemize}
\item Check that $M$ is a 3-manifold by verifying that the link of each vertex is a 2-sphere.
\item Verify that $M$ has the homology of a 3-sphere.  In particular, this implies that each 2-sphere in $M$ is separating.
\item Compute a maximal collection of disjoint  normal and almost normal 2-spheres in $M$. This can be done by solving the normal surface equations
and finding normal 2-spheres and almost normal 2-spheres among the fundamental solutions.  This follows Haken and reduces the search for such a family to a search within a Hilbert basis of solutions to the integer linear equations arising from normal surfaces \cite{Hak:61}.
\item  Cut open the manifold along a maximal collection of disjoint normal 2-spheres and examine each component in turn.
An easy topological argument  tells us that $M$ is homeomorphic to a 3-sphere if and only if every component is homeomorphic to a punctured 3-ball.
\item Components with two or more normal boundary 2-spheres are homeomorphic to punctured 3-balls.
\item Components with a single normal 2-sphere on their boundary are homeomorphic to a 3-ball if and only if they contain an almost normal 2-sphere or are neighborhoods (stars) of a vertex.
\item  $M$ is a 3-sphere if every component with a single normal 2-sphere on its boundary contain an almost normal 2-sphere or is a vertex neighborhood. 
 \end{itemize}
 
The structure of the algorithm is very similar to the 2-sphere characterization
described above. The characterization of the various complementary regions is also similar to that in dimension two. The evolution of a curve
by curvature is replaced by a normalization procedure in which a surface deforms to become normal or almost normal.  Thompson saw here that a correspondence between almost normal spheres and bridge 2-spheres in thin position allows the techniques of thin position to be used to establish the existence of almost normal spheres containing an octagonal disk \cite{Thompson}. \\ \\
{\bf Remark.}  There are differences between the characterizations used in the smooth and PL settings.  In the smooth setting, a region bounded entirely by stable minimal 2-spheres and containing no minimal 2-spheres in its interior cannot be a punctured ball.  An unstable minimal 2-sphere always exists in its interior.  In contrast, a region in a triangulated 3-manifold bounded entirely by normal 2-spheres and containing no normal 2-spheres in its interior is always a punctured ball.

 \section{Conclusion}

Rubenstein's  work on the existence of minimal surfaces in 3-manifolds and on PL-minimal surface theory naturally led him to the concept of an almost normal surface.  Almost normal surfaces are now widely recognized as  powerful tools  to apply in multiple areas of 3-manifold theory. 

Table~\ref{dictionary} summarizes some correspondences between the worlds of Riemannian manifolds  with their minimal submanifolds and of triangulated manifolds with their normal and almost normal submanifolds.

\begin{table} [htbp] 
\caption{Minimal Surface - Normal Surface Correspondences}
\label{dictionary}     % For LaTeX tables use
\begin{tabular}{| l| l |}
 \hline 
 Smooth Riemannian Manifolds & Combinatorial Triangulated Manifolds   \\
 \hline
  \hline
Geodesic    & Normal curve  \\
  \hline
  Length or Area & Weight  \\
  \hline
Stable minimal surface  & Normal surface \\
  \hline
Unstable minimal surface  & Almost normal surface\\
  \hline
Flow by mean curvature  & Normalization\\
  \hline
A smooth $S^3$ contains an unstable minimal $S^2$ &  A  PL $S^3$ contains an almost normal $S^2$  \\
  \hline
$  \partial X  $ a stable $S^2 $  and  int($X$) contains  & $ \partial  X   $  a normal $S^2 $   and   int($X$) contains \\
an unstable $S^2,$ no stable  $S^2$ &  an almost normal  $S^2$, no normal  $S^2$   \\
$ \implies  X= B^3$  & $\implies  X= B^3$  \\
\hline
\end{tabular}
\end{table}


\begin{thebibliography}{99} 
  
\bibitem{Birkhoff} 
G. Birkhoff,  Dynamical Systems, AMS, 1927. 


\bibitem{Croke} 
C. Croke,  {\em PoincarŽ's problem and the length of the shortest closed geodesic on a convex hypersurface,}  J. Differential Geom.  17, Number 4 (1982), 595-634.
 
 \bibitem{Dehn:10}
M. Dehn,
{\em \"Uber die Topologie des dreidimensional Raumes} Math. Annalen, 
{\bf 69} (1910), 137--168.

\bibitem{Grayson} 
M. Grayson,  {\em The heat equation shrinks embedded plane curves to round points,}  J. Differential Geom. 26, Number 2 (1987), 285-314.

\bibitem{Grayson2} 
M. Grayson,  {\em Shortening embedded curves,} Ann. of Math. 129 (1989) 71- 111.

\bibitem{Hak:61}
W. Haken,
``Theorie der Normalfl\"achen:
Ein Isotopiekriterium f\"ur den Kreisknoten'', {\em Acta Math.}, 105 (1961)
245--375.

\bibitem{HassScott} J. Hass and G.P. Scott, {\em Shortening curves on surfaces},
  Topology 33, (1994) 25-43.

\bibitem{HLP}
J. Hass, J. C. Lagarias and N. Pippenger,
\textit{The computational complexity of Knot and Link problems}, 
Journal of the ACM, 46, (1999) 185--211 .

\bibitem{JacoRubinstein} 
W. Jaco and  J.H. Rubinstein,  {\em PL minimal surfaces in 3-manifolds,}  J. Differential Geom. 27 (1988) 493--524.

\bibitem{Jost} 
J. Jost,  {\em Embedded minimal surfaces in manifolds diffeomorphic to the three-dimensional ball or sphere,}  
J. Differential Geom.  30, (1989), 555-577.
 
 \bibitem{Klingenberg}
 W.P.A. Klingenberg, Closed geodesics on Riemannian manifolds, CBMS, Washington, DC, 1983.
 
\bibitem{Knes:29}
H. Kneser,
``Geschlossene Fl\"achen in dreidimensionalen Mannigfaltigkeiten'',
{\em Jahresbericht Math. Verein.}, 28 (1929) 248--260.


\bibitem{MeeksSimonYau} 
W.  Meeks, III, L. Simon, and S.T. Yau  {\em Embedded minimal surfaces, exotic spheres, and manifolds with positive ricci curvature,}  
Annals of Mathematics,  2(3) (1982), 621--659.
 
\bibitem{Nabutovsky}  
A. Nabutovsky, {\em Einstein structures: Existence versus uniqueness,} Geom. Funct. Anal., 5(1):76Ð91, 1995.

\bibitem{Perelman}  
G. Perelman, {\em Finite extinction time for the solutions to the Ricci flow on certain three-manifolds,} arXiv:math.DG/0307245 [math.DG].
(2003).

\bibitem{PittsRubinstein} J.T. Pitts  and  J.H. Rubinstein, {\em Existence of minimal surfaces of bounded topological type in three-manifolds,} in Miniconference on geometry and partial differential equations (Canberra, 1985), 163Ð176; Proc. Centre Math. Anal. Austral. Nat. Univ., 10, Austral. Nat. Univ., Canberra, 1986.

\bibitem{Poincare} 
H. Poincare,  {\em Sur les lignes gŽodŽsiques sur les surfaces convexes, }  Trans. Amer. Math. Soc. 17, 237-274 ( 1909).
  
\bibitem{Rubinstein:1994} 
J.H. Rubinstein,  {\em An algorithm to recognize the 3-sphere, }  Proceedings of the International Congress of Mathematicians.,   
Vol. 1, 2 (Zurich, 1994), 601Ð-611; Birkhauser, Basel, 1995
 
 \bibitem{Rubinstein:1995} 
 J.H. Rubinstein,  {\em Polyhedral minimal surfaces, Heegaard splittings and decision problems for 3-manifolds, }  Proceedings of the Georgia Topology Conf., William H. Kazez, ed., International Press 1995, 1--20.

 \bibitem{Schoen}  
R. Schoen,   {\em Estimates for stable minimal surfaces in three-dimensional manifolds,} Seminar on minimal submanifolds,  Ann. of Math. Stud., vol. 103, Princeton Univ. Press, Princeton, NJ, 1983, 111--126.

 \bibitem{SimonSmith}  
F. Smith,  {\em  On the existence of embedded minimal 2-spheres in the 3Ðsphere, endowed
with an arbitrary Riemannian metric}, Ph.D. thesis, supervisor L. Simon, University of Melbourne
(1982).

 \bibitem{Stillwell}  
J. Stillwell, {\em The word problem and the isomorphism problem for groups,} Bull. AMS 6, (1982), 33--56.

 \bibitem{Thompson}  
 A. Thompson, {\em Thin position and the recognition problem for $S^3$}, Math. Res. Lett. 1
(1994), 613--630. MR 95k:57015

\bibitem{VolodinKuznetsovFomenko} 
 I.A. Volodin, V.E. Kuznetsov, and A.T. Fomenko, {\em The problem of discriminating algorithmically the standard three-dimensional sphere},  Usp. Mat. Nauk, 29(5):71--168, 1974. In Russian. 
 English translation: Russ. Math. Surv. 29,5:71--172 (1974).
 
\bibitem{White}  
B. White, {\em The space of minimal submanifolds for varying Riemannian metrics}, Indiana Math. J. 40, (1991), 161-200.
 
\end{thebibliography}
 \end{document}